\documentclass{ifacconf}
\usepackage{amsmath}
\usepackage{amssymb}
\usepackage{mathtools}

\usepackage{xcolor}
\usepackage{graphicx}
\usepackage{float} 
\usepackage{natbib}        
\usepackage{algpseudocode}
\usepackage{algorithm}
\usepackage{subcaption}
\usepackage{caption}

\allowdisplaybreaks[4]

\DeclareMathOperator{\softmax}{softmax}
\DeclareMathOperator{\LSE}{LSE}

\begin{document}
\begin{frontmatter}

\title{Successive~Convex~Optimization for Transformer~Encoder~Model~Predictive~Control}

\author[Oxford]{Xingxiao Chen} 
\author[Oxford]{Mark Cannon} 

\address[Oxford]{Department of Engineering Science, University of Oxford, \{xingxiao.chen, mark.cannon\}@eng.ox.ac.uk}

\begin{abstract}
We propose a data-driven Model Predictive Control (MPC) framework that employs a transformer encoder to generate multi-step predictions. To handle the nonconvex attention mechanism, we derive difference of convex (DC) representations of the transformer encoder components and embed them in a successive convex programming (SCP) iteration. Recursive feasibility and convergence of the SCP iterates are guaranteed, and each iterate yields a solution estimate satisfying the problem constraints. Under mild assumptions, the SCP iteration converges to a locally optimal solution of the MPC problem. The approach is illustrated on a benchmark nonlinear control problem. 
\end{abstract}


\end{frontmatter}

\section{Introduction}

Model Predictive Control (MPC) has proven to be an effective control strategy in diverse application areas. MPC computes a feedback control law by optimizing predicted performance subject to constraints in real-time \citep{kouvaritakis2015model,rawlings2017model}. A key requirement for MPC is a  model of the controlled system with a known degree of accuracy. However, in many  applications, highly nonlinear dynamics and unpredictable disturbances make it difficult to identify and deploy suitable models within an online optimization scheme. Recent advances in machine learning (ML) have introduced alternative ways to learn nonlinear model dynamics, provided sufficiently rich and accurate data on system responses are available \citep{NOROUZI2023105878}. This is particularly useful when a first-principles model is difficult to derive and when the system exhibits complex behaviour. 

The integration of ML and MPC is a rapidly developing research area. This paper focusses on a specific ML technique for representing system dynamics within the MPC framework. We consider learning models whose core is a transformer encoder architecture, which uses attention mechanisms to effectively capture sequential dependencies \citep{10.5555/3295222.3295349}. Previous studies mostly concentrate on using neural networks to predict system behaviour, without explicitly addressing the resulting optimization problem with a neural network serving as a predictor. Among optimization methods employed in this context, Sequential Quadratic Programming (SQP) \citep[e.g.][]{Nocedal2006}, is most commonly used. Motivated by the limitations of conventional nonlinear programming methods such as SQP when neural-network predictors are embedded in MPC, we instead formulate the resulting optimization problem through a DC modelling framework and solve it using the Convex-Concave Procedure (CCP) \citep{Lipp2016}.

Combined with a difference of convex (DC) modelling framework, this approach can handle nonconvex constraints and provide robustness guarantees \citep{9993390,ju_2,KRAUSCH20241699}. In this setting we develop a CCP algorithm tailored to transformer encoder-based predictors by constructing DC representations of the model's components. We discuss the convergence guarantees provided by this algorithm and modifications to the transformer encoder architecture to ensure its compatibility with the proposed optimization technique. 

The paper is organised as follows. Section~\ref{sec:RW} discusses previous work and literature relevant to this paper. Section~\ref{sec:Overview} gives a brief introduction to the MPC design. Section~\ref{sec:TFE} introduces the learning-based predictor architecture. Section~\ref{sec:CCP} demonstrates the modified method based on DC programming and Section~\ref{sec:converge} discusses its convergence properties. Section~\ref{sec:results} presents the experimental results while Section~\ref{sec:conclu} concludes the paper.

\section{Related work}\label{sec:RW}

Recent studies in process control have demonstrated the potential of replacing the prediction model in MPC with neural network models. \citet{wong2018recurrentneuralnetworkbasedmodel} propose MPC using recurrent neural networks (RNNs) to capture complex reaction kinetics in continuous pharmaceutical manufacturing. \citet{MACHALEK2022100172} demonstrate a gated recurrent unit (GRU)-based encoder-decoder structure for multi-step prediction in MPC for a thermal power plant. \citet{MASTI2021109666} present an autoencoder-based system identification method capable of learning Hammerstein-Wiener system dynamics. These approaches obtain multi-step predictions by iteratively applying a one-step model over a finite prediction horizon. However, iterating a one-step model for multi-step prediction has limitations in both accuracy and efficiency. \citet{PARK2023108396} explores a multi-step architecture based on transformers, leveraging their ability to generate predictions in parallel. Compared with commonly used RNNs, transformer architectures demonstrate superior accuracy in capturing long-term dependencies within data sequences due to their attention mechanism  \citep{10.5555/3295222.3295349}. 

The nonlinear structure of the attention mechanism complicates the MPC optimization problem when the predictor incorporates attention blocks. \citet{ZARZYCKI2022229} propose an approach based on Long Short-Term Memory (LSTM) and GRU using linear model approximations. However, approximate predictions based on linear models can result in inaccurate solution estimates and slow convergence of successive linearization iterations. \citet{JUNG2023106226} present an LSTM-based MPC formulation that uses automatic differentiation to provide derivatives to optimization solvers. In~\cite{PARK2023108396} and \cite{JUNG2023106226}, it is shown that existing solvers, such as SLSQP, can solve optimization problems with LSTM or transformer models acting as predictors. 

SLSQP can be sensitive to the initial solution estimate, and poor initialization may result in suboptimal solutions particularly when a transformer encoder block is involved due to its high degree of nonlinearity. Given the nonconvexity introduced by the attention mechanism, optimizing a trained transformer encoder block is challenging. \citet{ergen2022convexifyingtransformersimprovingoptimization} convexify the transformer encoder by substituting the attention block with linear layers, thereby casting the problem as a convex optimization task; however, this is intended to simplify the training process and leads to limited prediction accuracy. \cite{articleAwasthi} discuss how to apply DC programming and the CCP to convolutional neural networks. Inspired by this idea, and more generally by SCP applied to data-driven robust MPC \citep{9993390,KRAUSCH20241699,ju_2,ju_1}, we propose a method for optimizing transformer encoder-based predictors via DC programming, incorporating appropriate modifications to the attention block.

\section{Overview of data-driven MPC design}\label{sec:Overview}

\begin{figure}[b]
      \includegraphics[trim={37mm 55mm 37mm 31mm},clip,scale=0.335]{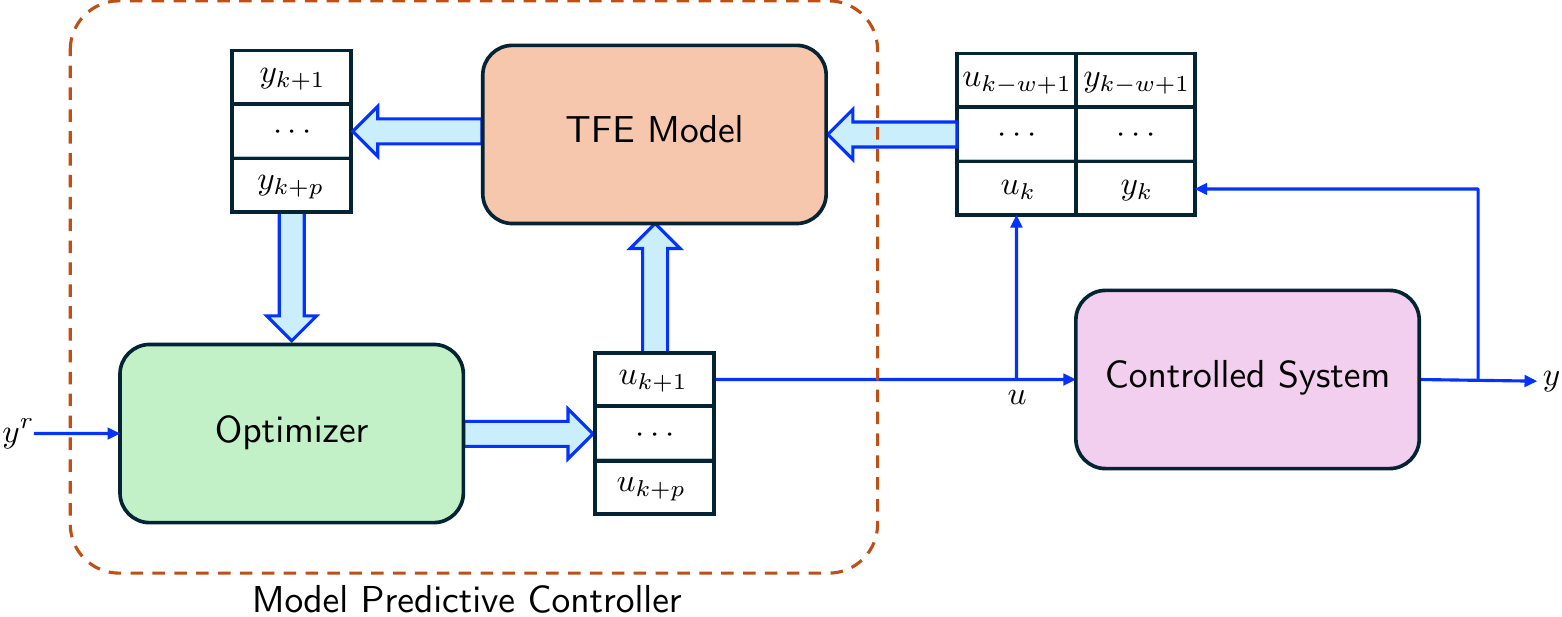}
      \caption{Overview of the MPC strategy with a transformer encoder (TFE)-based predictor.}
      \label{fig:DesignOverview}
   \end{figure}
   
This section briefly introduces the transformer encoder-based multi-step MPC strategy. 
Figure~\ref{fig:DesignOverview} illustrates the proposed closed-loop MPC scheme 
applied to a system with measured output $y_k$ and control input $u_k$ at time $k$. Here \( w \) represents the length of the past data sequence and \( p \) denotes the length of the future prediction window. The transformer encoder-based model uses the available past data (\( u_{k-w+1}, \ldots,u_{k}, y_{k-w+1}, \ldots,y_{k}  \)), along with anticipated future control inputs (\( u_{k+1}, \ldots,u_{k+p} \)) to predict future outputs (\( y_{k+1}, \ldots, y_{k+p} \)) of the controlled system. The vector of predicted outputs is passed to the optimizer, which determines the optimal control sequence (\( u_{k+1}, \ldots, u_{k+p} \)). The optimal control input $u_{k+1}$ is fed into the controlled system at time $k+1$, and, together with the measured system output \(y_{k+1}\), is fed back into learned model to update the model state for time \( k+1 \). 

\section{Transformer encoder-based Predictor}\label{sec:TFE}

\subsection{Data preparation and pre-processing} 
To illustrate our approach, we consider a nonlinear coupled $2$-tank system as a case study. While the system has two states, which correspond to the water level in each tank, our aim is to predict the water level of tank 2, reducing the problem to a single-input single-output system. Let \( x = [\begin{matrix} x_1 & x_2 \end{matrix}]^\top \) be the state vector, where \( x_1 \) and \( x_2 \) represent the water levels in tanks 1 and 2 respectively. The input (pump control signal) is denoted \( u\). The system with discrete time index $k$ and  parameters in Table~\ref{tab:twotank} 
is as follows:
\begin{equation}
\label{fun:siso}
  \begin{aligned}
    x[k+1] &= f(x[k], u[k]) + w[k]\\
    f_1(x, u) &= x_1 + \frac{\delta}{A} \left( k_p u - A_1 \sqrt{2 g x_1} \right) \\
    f_2(x, u) &= x_2 + \frac{\delta}{A} \left( A_1 \sqrt{2 g x_1} - A_2 \sqrt{2 g x_2} \right) 
  \end{aligned}
\end{equation}
%
The water level, \(x_1\), of tank 1 is an unmeasured state, and to simplify notation in the following development we use \(y_k\) to denote the level $x_2$ of tank 2 at time $k$. 
The structure of the input-output training data pairs is given by
\[
     X = \left[~\begin{matrix} 
     u_{k-w+1} & y_{k-w+1}
     \\ 
     \vdots & \vdots
     \\ 
     u_{k} & y_{k}
     \\
     \hline
     U & {\bf 0}_p\rule{0pt}{10pt}
     \end{matrix}~\right].
     \ \ 
     Y =\begin{bmatrix}
             y_{k+1}\\
             \vdots \\
             y_{k+p}
           \end{bmatrix}
 \]
with $U = [u_{k+1} \, \cdots \, u_{k+p}]^\top$ and ${\bf 0}_p = [0 \, \cdots \, 0]^\top\in\mathbb{R}^p$.

\begin{table}[h]
\caption{Two-tank model parameters}
\label{tab:twotank}
\begin{center}
\resizebox{\columnwidth}{!}{
\begin{tabular}{c c c c}
\hline
\textbf{Parameter} & \textbf{Symbol} & \textbf{Value} & \textbf{Unit}\\
\hline
Gravitational acceleration & \( g \) & 981 & \(cm \cdot s^{-2}\)\\
\hline
Time step & \(\delta\) & 1 & \(s\)\\
\hline
Pump gain & \(k_p\) & 3.3 & \(cm^3  s^{-1}  V^{-1}\) \\
\hline
Cross-sectional area & \(A\) & 15.2 & \(cm^2\)\\
\hline
Outflow orifice areas  & \(A_1, A_2\) & 0.135, 0.140 & \(cm^2\)\\
\hline
\end{tabular}}
\end{center}
\end{table}

\subsection{Structure of the prediction model}

\begin{figure}[thpb]
\centering  
\includegraphics[trim={101mm 14mm 59mm 30mm},clip,scale=0.43]{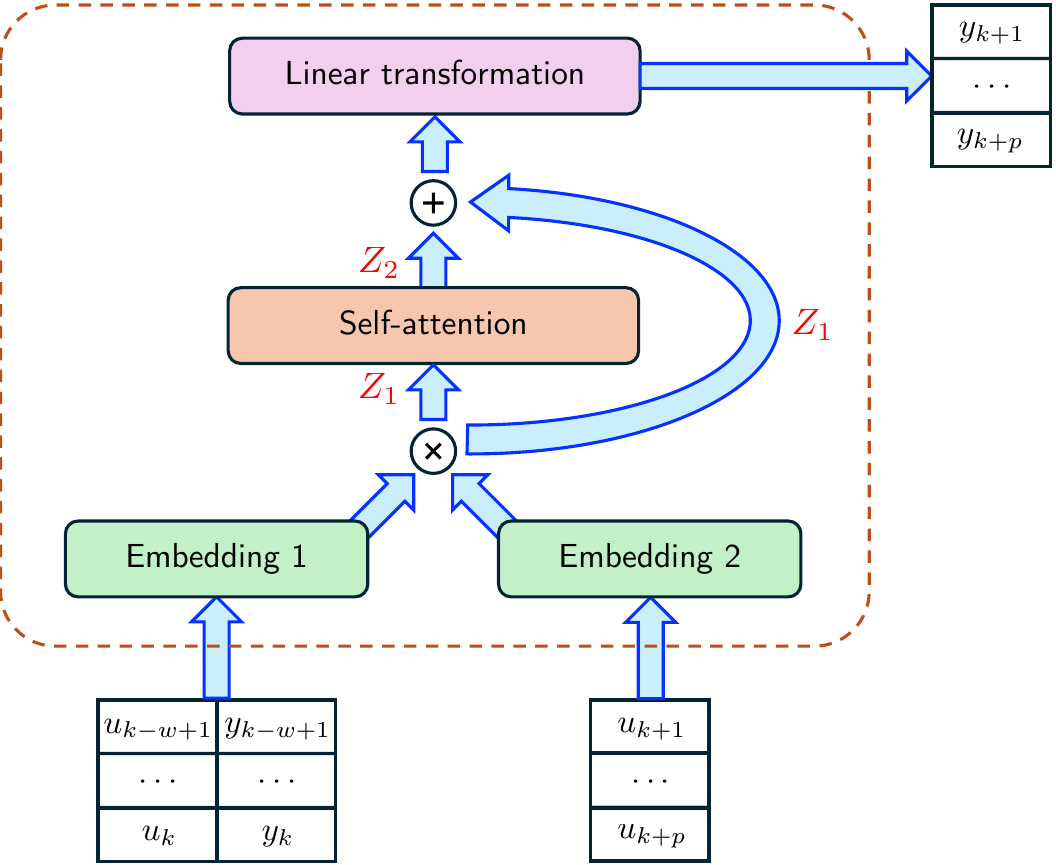} 
\caption{Structure of the transformer encoder-based model}
\label{fig:transformerencoderbasedModel}
\end{figure}

Figure~\ref{fig:transformerencoderbasedModel} illustrates in detail the TFE model appearing in Figure~\ref{fig:DesignOverview}. In the proposed structure, layer normalization calculations have been removed. Meanwhile, residual connections are retained to mitigate issues related to the lack of normalization by facilitating smoother gradient flow throughout the network \citep{he2016deep}. Self-attention is used in this preliminary study, with the understanding that our approach can be extended to multi-head attention architectures. Feature slicing is applied before inputting \(X\) into the  embedding layers. This method helps to split the data into past and future segments, allowing \textit{Embedding1} to focus on capturing the sequential relationships inherent in the past data. According to Fig.~\ref{fig:transformerencoderbasedModel}, the first computational stage of the model is defined as follows:
\begin{equation}
\label{fun:Z1}
    Z_1 = G(U) = \begin{bmatrix}
        X_pW_{EM1}^\top + \mathbf{1}b_{EM1}^\top \\
        U W_{EM2}^\top + \mathbf{1}b_{EM2}^\top
        \end{bmatrix} ,
\end{equation}
where \(U \in \mathbb{R}^{p}\) is the anticipated future input sequence and \(X_p\in \mathbb{R}^{w \times 2}\) contains the first $w$ rows of $X$, consisting only of past data. 
Here \(W_{EM1} \in \mathbb{R}^{d_k \times 2}\), \(W_{EM2} \in \mathbb{R}^{d_k \times 1}\) are weight matrices, \(b_{EM1} \in \mathbb{R}^{d_k}\) and \(b_{EM2} \in \mathbb{R}^{d_k}\) are bias vectors for embedding layers 1 and 2 respectively, and
\(d_k\) is a pre-set value representing the dimensionality of the keys in the attention mechanism. Taking \(Z_1 \in \mathbb{R}^{(w+p) \times d_k}\) as the input to the self-attention block, the attention score \(Z_2 \in \mathbb{R}^{(w+p) \times d_k}\) is defined as: 
\begin{equation}
\label{fun:Z2}
  \begin{aligned}
    Z_2 = \softmax(\frac{QK^\top}{\sqrt{d_k}})V,
  \end{aligned}
\end{equation}
and \(Q \in \mathbb{R}^{(w+p) \times d_k}\), \(K \in \mathbb{R}^{(w+p) \times d_k}\), \(V \in \mathbb{R}^{(w+p) \times d_k}\) are the queries, keys and values in the attention mechanism
\begin{equation}
\label{fun:QKV}
  \begin{aligned}
    Q &= Z_1W_Q^\top + \mathbf{1}b_Q^\top \\
    K &= Z_1W_K^\top + \mathbf{1}b_K^\top \\
    V &= Z_1W_V^\top + \mathbf{1}b_V^\top.
  \end{aligned}
\end{equation}
The weight matrices are denoted  \(W_{Q} \in \mathbb{R}^{d_k \times d_k}\), \(W_{K} \in \mathbb{R}^{d_k \times d_k}\), \(W_{V} \in \mathbb{R}^{d_k \times d_k}\), and $\mathbf{1} \in \mathbb{R}^{(w+p)}$ is a vector of ones. The bias vectors are represented as \(b_{Q} \in \mathbb{R}^{d_k}\), \(b_{K} \in \mathbb{R}^{d_k}\), \(b_{V} \in \mathbb{R}^{d_k}\). Finally, the output \(Y \in \mathbb{R}^{p}\) is given by 
\begin{equation}
\label{fun:Y}
  \begin{aligned}
    Y = W_{L2} \bigl((Z_1 + Z_2 )W_{L1} + b_{L1}\bigr) + b_{L2}
  \end{aligned}
\end{equation}
where \(W_{L1} \in \mathbb{R}^{d_k}\), \(W_{L2} \in \mathbb{R}^{p \times (w+p)}\) are the weight matrices of the linear transformation, and \(b_{L1} \in \mathbb{R}^{w+p}\) and \(b_{L2} \in \mathbb{R}^{p}\) are the associated bias vectors.

To evaluate the proposed model, we test various combinations of \(w\) and \(p\) with prediction accuracy \citep{MASTI2021109666} evaluated using the best fit ratio (BFR):
\begin{equation}
\label{fun:BFR}
  \begin{aligned}
    \mathit{BFR} = \max(0,1-\frac{\lVert Y-Y^{\mathrm{true}} \rVert_2}{\lVert Y^{\mathrm{true}} - \overline{Y} \rVert_2}) ,
  \end{aligned}
\end{equation}
where \(Y^{\mathrm{true}}\) is the system output with mean \(\overline{Y}\), and \(Y\) is the output predicted by the model. Based on the results in Table.~\ref{tab:BFR}, we choose $p=20$ because this provides the best tradeoff between capturing the delayed dynamics of the plant, while maintaining the accuracy of the prediction and keeping the optimization computationally practical.

\begin{table}[t]
\caption{Open-loop prediction performance}
\label{tab:BFR}
\begin{center}
\begin{tabular}{c c c c c}
\hline
\textbf{$w$} & \textbf{$p$} & \textbf{Training BFR} & \textbf{Validation BFR}\\
\hline
8 & 10 & 98.42\(\%\) & 94.93\(\%\)\\
\hline
8 & 20 & 98.67\(\%\) & 95.74\(\%\)\\
\hline
8 & 30 & 97.52\(\%\) & 92.40\(\%\)\\
\hline
8 & 40 & 97.43\(\%\) & 92.23\(\%\) \\
\hline
\end{tabular}
\end{center}
\end{table}

\section{Convex-concave procedure}\label{sec:CCP}

We propose the application of DC programming to both the attention mechanism and the fully connected layer. By employing the DC framework, we can decompose the nonconvex problem into more manageable convex subproblems. The attention mechanism is divided into three stages, taking into account the multiplication of the query \(Q\) and key \(K\), the use of the softmax function, and the subsequent multiplication of the softmax output with the value \(V\). The complete decomposition is as follows.

\subsection{Self-attention stage 1}\label{sec:stage1}
We first construct convex functions \( P_1(Z_1)\) and \( P_2(Z_1)\) such that \(QK^\top = P_1(Z_1) - P_2(Z_1)\).
Let $P(Z_1) = Q K^\top$, then from (\ref{fun:QKV}) we obtain
\begin{align}
P(Z_1) 
&= 
Q(Z_1) K(Z_1)^\top \notag\\
&\quad = 
(Z_1W_Q^\top + \mathbf{1} b_Q^\top)(Z_1W_K^\top + \mathbf{1} b_K^\top )^\top 
\label{fun:PZ1} 
\end{align}
Denoting the $i$th and $j$th rows of $Q$ and $K$ by $Q_i\in\mathbb{R}^{d_k}$ and $K_j\in\mathbb{R}^{d_k}$ respectively, the $(i,j)$th element of $P$ is therefore
\begin{equation}
P_{ij}(Z_1) = Q_i(Z_1) ^\top K_j(Z_1)
\quad i,j = 1,\dots,w+p.
\label{eq:P-entry-inner-product}
\end{equation}
In general, this a nonconvex quadratic function,
and to obtain a DC representation we use the following identity, which expresses $x^\top y$ as a difference of convex quadratic functions with minimum curvature\footnote{In \eqref{fun:bilinear-identity}, $x^\top y = p_1(x,y) - p_2(x,y)$ where $p_1,p_2$ are the convex quadratic functions that have Hessian matrices with the smallest possible maximum eigenvalue.}
for $x,y \in \mathbb{R}^{d_k}$:
\begin{equation}
  x^\top y
  \;=\;
  \tfrac{1}{4}\,\bigl\|x + y\bigr\|_2^2
  \;-\;
  \tfrac{1}{4}\,\bigl\|x - y\bigr\|_2^2 .
  \label{fun:bilinear-identity}
\end{equation}
Applying \eqref{fun:bilinear-identity} with $x = Q_i(Z_1)$ and $y = K_j(Z_1)$ yields
\begin{equation}
P_{ij}(Z_1)  = [P_1(Z_1)]_{ij} - [P_2(Z_1)]_{ij},
\end{equation}
where $P_1$ and $P_2$ are convex functions
\begin{align*}
[P_1(Z_1)]_{ij} &= \tfrac{1}{4}\|Q_i(Z_1) + K_j(Z_1)\|_2^2
\\
[P_2(Z_1)]_{ij} &= \tfrac{1}{4}\|Q_i(Z_1) - K_j(Z_1)\|_2^2.
\end{align*}
This decomposition enables the CCP to be applied to optimization problems involving $P(Z_1)$. In particular, linearizing \(P_1(Z_1)\) or \(P_2(Z_1)\) at \(Z_1 = Z_1^k\), we obtain a concave lower bound and convex upper bound on \(P(Z_1)\) respectively:
\begin{align}
&\underline{P}_{ij}(Z_1) = \tfrac{1}{4}\bigl\|Q_i(Z_1^k) + K_j(Z_1^k)\bigr\|_2^2  - \tfrac{1}{4}\bigl\|Q_i(Z_1) - K_j(Z_1)\bigr\|_2^2 
\notag\\ 
&\quad + \tfrac{1}{2}(Q_i(Z_1^k) + K_j(Z_1^k))^\top [W_Q (Z_1-Z_1^k)^\top]_i
\notag\\ 
&\quad + \tfrac{1}{2}(Q_i(Z_1^k) + K_j(Z_1^k))^\top [W_K (Z_1-Z_1^k)^\top]_j
\leq P _{ij} (Z_1)
\label{fun:Pboundslow}
\\
&\overline{P}_{ij}(Z_1) = \tfrac{1}{4}\bigl\|Q_i(Z_1) + K_j(Z_1)\bigr\|_2^2  - \tfrac{1}{4}\bigl\|Q_i(Z_1^k) - K_j(Z_1^k)\bigr\|_2^2 
\notag\\ 
&\quad - \tfrac{1}{2}(Q_i(Z_1^k) - K_j(Z_1^k))^\top [W_Q (Z_1-Z_1^k)^\top]_i 
\notag\\
&\quad + \tfrac{1}{2}(Q_i(Z_1^k) - K_j(Z_1^k)) ^\top [W_K (Z_1-Z_1^k)^\top]_j
\geq P_{ij} (Z_1)
\label{fun:Pboundsup}
\end{align}
where $[W_Q (Z_1 - Z_1^k)^\top]_i$ and $[W_K (Z_1 - Z_1^k)^\top]_j$ denote the $i$th and $j$th columns of the respective matrix products. 
%


\subsection{Self-attention stage 2}\label{sec:stage2}
We now consider convex bounds on the softmax function applied to \(P \in \smash{\mathbb{R}^{(w+p) \times (w+p)}}\). Let  
$R = \softmax(P/\sqrt{d_k})$, where $R\in \smash{\mathbb{R}^{(w+p) \times (w+p)}} $, and let $R_n$ and $P_n$ denote the $n$th rows of $R$ and $P/\sqrt{d_k}$ with $i$th components $R_{ni}$ and $P_{ni}$, respectively. 
Then the softmax function defines 
$R_{ni}$ as
\begin{equation}
\label{fun:RPi}
    R_{ni}(P_{n}) = \frac{\exp(P_{ni})}{\sum_{j=1}^{w+p} \exp(P_{n,j})} ,
\qquad i = 1,\ldots,w+p. 
\end{equation}
Consider the log-softmax function given by
\[
 \ln (R_{ni}(P_{n})) = P_{ni}-\ln(\sum_{j=1}^{w+p} \exp(P_{n,j})) .
\]
The second term on the RHS is the log-sum-exp (LSE) function:
$\LSE(P_{n}) = \ln(\sum_{j=1}^{w+p} \exp(P_{n,j}))$, which is convex in $P_n$.
%
Let \(P(Z_1^k) = P^k\), where \(Z_1^k\) is the linearization point used in (\ref{fun:Pboundsup})-(\ref{fun:Pboundslow}). Then, since $\LSE()$ is convex, the Jacobian of $\LSE(P_n)$ at $P^k$ yields a linear upper bound:
\[
\ln (R_{ni}(P_{n})) \!\leq\! P_{ni}-\LSE(P_{n}^k) 
-\Bigl[\frac{\exp(P_{n}^k)}{\sum_{j =1}^{w+p} \!\exp(P_{n,j}^k)}\Bigr]^{\!\top}\!\!\!(P_{n}-P_{n}^k).
\]
The monotonicity of $\exp()$ implies 
$R_{ni}(P_n)\leq \overline{R}_{ni}(P_{n})$ for all $P_n\in \mathbb{R}^{w+p}$
where 
\begin{align}
\overline{R}_{ni}(P_{n}) &= 
\exp \Bigl\{ P_{ni} - \LSE(P_{n}^k) 
\nonumber \\
&\quad -\Bigl[\frac{\exp(P_{n}^k)}{\sum_{j =1}^{w+p} \exp(P_{n,j}^k)}\Bigr]^{\top}(P_{n}-P_{n}^k)\Bigr\},
\label{eq:softmax_upperbnd}
\end{align}
and furthermore $\overline{R}_{ni}()$ is convex.

To obtain a concave lower bound on $R = \softmax(P)$ we use the following result.

\begin{lem}\label{lem:softmax_symmetry}
For all $P_n\in \mathbb{R}^{w+p}$, $R_{ni}(P_n) \geq \underline{R}_{ni}(P_n)$ where
\begin{align}
\underline{R}_{ni}(P_n) &= 1 - \sum_{j\neq i}^{w+p} \exp(P_{nj}) \exp\Bigl\{-\LSE(P_n^k) 
\nonumber \\
&\quad - 
\Bigl[\frac{\exp(P_{n}^k)}{\sum_{j =1}^{w+p} \!\exp(P_{nj}^k)}\Bigr]^{\!\top}\!\!\!(P_{n}-P_{n}^k)
\Bigr\},
\label{eq:softmax_lowerbnd}
\end{align}
and furthermore $\underline{R}_{ni}()$ is concave.
\end{lem}

\begin{pf}
To derive \eqref{eq:softmax_lowerbnd} we rewrite (\ref{fun:RPi}) as
\begin{align*}
R_{ni}(P_{n}) 
&= 1 - \frac{\sum_{j\neq i}^{w+p} \exp(P_{nj})}{\sum_{j=1}^{w+p} \exp(P_{nj})} 
\\
&= 1 - \exp\Bigl[ \ln \sum_{j\neq i}^{w+p} \exp(P_{nj}) - \LSE(P_{n}) \Bigr],
\end{align*}
then replace $\LSE(P_{n})$ on the RHS with its Jacobian linear approximation around $P^k$. Since term in square brackets on the RHS is a difference of  convex functions, the resulting function is a lower bound on $R_{ni}(P_{n})$ and is concave in $P_n$ due to the monotonicity of $\exp()$.
\qed\end{pf}



We can now state convex conditions imposing upper and lower bounds on the softmax function  \eqref{fun:RPi}.

\begin{thm}
For any $P^k$, the constraints
\begin{equation}\label{eq:softmax_bnds}
\underline{r}_{ni} \leq \underline{R}_{ni}(P_n), 
\quad
\overline{R}_{ni}(P_n) \leq \overline{r}_{ni},
\end{equation}
are convex in $P_n$ and sufficient for $\underline{r}_{ni} \leq R_{ni}(P_n) \leq \overline{r}_{ni}$, and necessary for $\underline{r}_{ni} \leq R_{ni}(P_n^k) \leq \overline{r}_{ni}$.
\end{thm}

\begin{pf}
Sufficiency of \eqref{eq:softmax_bnds} for $\underline{r}_{ni} \leq R_{ni}(P_n) \leq \overline{r}_{ni}$ 
follows from $\underline{R}_{ni}(P_{n}) \leq R_{ni}(P_n)\leq \overline{R}_{ni}(P_{n})$ for all $P_n\in \mathbb{R}^{w+p}$,
while convexity of \eqref{eq:softmax_bnds} 
is due to the convexity of $\overline{R}_{ni}()$ and the concavity of $\underline{R}_{ni}()$.
Necessity of \eqref{eq:softmax_bnds} for $\underline{r}_{ni} \leq R_{ni}(P_n^k) \leq \overline{r}_{ni}$ is a consequence of the fact that Jacobian linearization provides a tight bound on a convex function, and hence $\underline{R}_{ni}(P_n^k)=R_{ni}(P_n^k)=\overline{R}_{ni}(P_n^k)$.
\qed\end{pf}

\subsection{Self-attention stage 3}\label{sec:stage3}

The output of the self-attention block is $Z_2 =RV$ where $R=\softmax(P/\sqrt{d_k})$ and $V= Z_1 W_V^\top + \mathbf{1}b_V^\top$ from (\ref{fun:Z2}) and (\ref{fun:QKV}). We denote $Z_2 = S(R,Z_1)$ and express the DC decomposition of the $(i,j)$th element of $S(R,Z_1)$ as
\begin{align*}
S_{ij}(R,Z_1) &= 
\tfrac{1}{4} \| R_i + [Z_1V_K^\top + \mathbf{1} b_V^\top ]_j \|^2_2
\\
&\quad - \tfrac{1}{4} \| R_i - [Z_1W_V^\top +\mathbf{1} b_V^\top ]_j \|^2_2
\end{align*}
where $R_i \in\mathbb{R}^{w+p}$ and $[Z_1W_V^\top +\mathbf{1} b_V^\top ]_j \in\mathbb{R}^{w+p}$ are the $i$th row and $j$th column of $R$ and $Z_1W_V^\top +\mathbf{1} b_V^\top$ respectively.
By Jacobian linearization around \(R^k, Z_1^k\), we obtain the following quadratic bounds on \(S_{ij}(R,Z_1)\):
\begin{align}
    & \underline{S}_{ij}(R,Z_1) = \tfrac{1}{4} \| R_i^k + [Z_1^kW_V^\top + \mathbf{1} b_V^\top ]_j \|^2_2 
\notag\\ 
    &\quad -\tfrac{1}{4} \|R_i - [Z_1W_V^\top + \mathbf{1} b_V^\top]_j \|^2_2 
\notag\\ 
    &\quad + \tfrac{1}{2}(R_i^k + [Z_1^kW_V^\top + \mathbf{1} b_V^\top]_j)^\top (R_i-R_i^k) 
\notag\\
    &\quad + \tfrac{1}{2}(R_i^k + [Z_1^kW_V^\top + \mathbf{1} b_V^\top]_j)^\top [ (Z_1-Z_1^k) W_V^\top]_j
\label{fun:Sboundlow}
\\
    &\overline{S}_{ij}(R,Z_1) = \tfrac{1}{4} \| R_i + [Z_1W_V^\top + \mathbf{1} b_V^\top]_j \|^2_2 
\notag\\ 
    &\quad - \tfrac{1}{4} \| R^k_i - [Z_1^kW_V^\top + \mathbf{1} b_V^\top ]_j \|^2_2 
\notag\\ 
    &\quad - \tfrac{1}{2}  (R_i^k - [Z_1^kW_V^\top + \mathbf{1} b_V^\top]_j  )^\top
 (R_i- R_i^k)
\notag\\
    &\quad + \tfrac{1}{2}( R_i^k - [Z_1^kW_V^\top + \mathbf{1} b_V^\top]_j )^\top  [(Z_1-Z_1^k) W_V^\top]_j
\label{fun:Sboundup}
\end{align}
Thus $\underline{S}(R,Z_1) \le S(R,Z_1) \le \overline{S}(R,Z_1)$, and furthermore $\underline{S}(R,Z_1)$ and $\overline{S}(R,Z_1)$ are respectively concave and convex functions of $(R,Z_1)$, where $\underline{S}$ and $\overline{S}$ are matrices with $(i,j)$th elements $\underline{S}_{ij}$ and $\overline{S}_{ij}$, respectively.
These functions provide bounds on the intermediate variable $Z_2$ in (\ref{fun:Z2}), and hence provide bounds on the output $Y$ in
(\ref{fun:Y}). 

\subsection{CCP algorithm}

Bounds on the tracking objective are
computed over the set of possible outputs $Y$ consistent with the elementwise bounds constructed for the intermediate variables $(P,R,S)$ in self-attention stages 1-3.
However, enumerating the extreme points corresponding to these elementwise bounds is impractical (unless $d_k$, $w$ and $p$ are very small), so we instead bound the search space within a simplex and minimize the worst-case tracking error over this set.

At each iteration of the CCP algorithm, the model components are convexified using the solution estimate generated by the previous iteration, producing a sequence of subproblems that can be solved using a convex solver such as MOSEK \citep{mosek}. Algorithm~\ref{alg:MCCPA} summarizes the optimization problem. The constraints arise from the need to propagate uncertainty through the sequence of internal representations
$(P, R, S)$ within the TFE model. 

Each stage of this model involves a nonconvex mapping, and therefore the CCP bounds the output of each stage using a
convex set computed using convex conditions. 
%
Let $\mathcal{D}(0,1)$ denote the standard simplex in $\mathbb{R}^n$:
\[
  \mathcal{D}(0,1)
  =
  \mathrm{co} \{ 0, e_1, \dots, e_n \},
\]
where $e_i$ is the $i$th canonical basis vector and $\mathrm{co}\{\}$ denotes the convex hull.
Then for any $(\alpha,\beta)\in \mathbb{R}^{n+1}$ such that $\beta + {\bf 1}^\top\alpha \geq 0$, the simplex $\mathcal{D}(0,1)$ after scaling by $\beta+{\bf 1}^\top \alpha$ and translation through $-\alpha$, is given by
\[
\mathcal{D}(\alpha,\beta) = \{ x\in \mathbb{R}^n : -\alpha \leq x, \ {\bf 1}^\top x \leq \beta\} .
\]
The admissible internal variables are thus subject to the constraints
$P\in\mathcal{C}(\alpha_1,\beta_1)$,
$R\in\mathcal{C}(\alpha_2,\beta_2)$, 
$S\in\mathcal{C}(\alpha_3,\beta_3)$,
where $\alpha_1,\alpha_2,\alpha_3$ and $\beta_1,\beta_2,\beta_3$ are optimization variables
and where $X\in \mathcal{C}(\alpha,\beta)$ if and only if $x\in \mathcal{D}(\alpha,\beta)$, where $x$ is the vectorization of the matrix $X$.
%
The output $Y$ is computed
through the affine transformation:
\[
  Y = W_{L2} \bigl((Z_1 + S)W_{L1}^\top + b_{L1}\bigr) + b_{L2} .
\]
The iteration in Algorithm~1 (lines 4-9) solves a convex optimization problem using the convex bounds derived in \eqref{fun:Pboundslow}-\eqref{fun:Pboundsup}, \eqref{eq:softmax_upperbnd}-\eqref{eq:softmax_lowerbnd}, and \eqref{fun:Sboundup}-\eqref{fun:Sboundlow}. The linearization points employed in these bounds are defined by the solution of the previous iteration. As we show in Section~\ref{sec:converge}, this ensures convergence of the iteration to a local optimum.

\begin{algorithm}[t]
\caption{CCP Algorithm TFE-based MPC}\label{alg:MCCPA}
\begin{algorithmic}[1]
\State \textbf{Input:} Initial guess $U^{0}$, tolerance $\varepsilon > 0$, maximum iterations $K_{\max}$.
\State \textbf{Variables:} 
$U,\alpha_1,\alpha_2,\alpha_3,\beta_1,\beta_2,\beta_3$.
\State Set $k \gets 0$.
\Repeat
    \State \textbf{Linearization point:} Evaluate the model at $U^k$ and set $Z_1^{k}\gets G(U^k)$, 
           $P^k\gets P(Z_1^{k})$, $R^{k} \gets R(P^{k})$, and $S^k\gets S(R^{k}, Z_1^{k})$.
    \State \textbf{DC convexification:} 
           Construct convex bounds
           $\underline{P}(Z_1)$, $\overline{P}(Z_1)$, 
           $\underline{R}(P)$, $\overline{R}(P)$, and
           $\underline{S}(R,Z_1)$, $\overline{S}(R,Z_1)$
           using
           \eqref{fun:Pboundslow}-\eqref{fun:Pboundsup}, \eqref{eq:softmax_upperbnd}-\eqref{eq:softmax_lowerbnd}, \eqref{fun:Sboundup}-\eqref{fun:Sboundlow}
           with $Z_1^{k},P^{k}, R^{k}, S^k$.
    \State \textbf{Convex subproblem:}
           Solve
    \begin{align*}
      %
      & J^{k} = \min_{\substack{U,\alpha_1,\alpha_2,\alpha_3\\\beta_1,\beta_2,\beta_3}} \  \max_{S\in \mathcal{C}(\alpha_3, \beta_3)}
      \|Y - Y^r\|^2 + \lambda \|\Delta U\|^2
      \\
      & \text{s.t.} \ 
      Y = W_{L2}\bigl((Z_1 + S)W_{L1}^\top + b_{L1}\bigr) + b_{L2}^\top
      \\
      &\hspace{1ex}
      \begin{alignedat}{2}
      &{-\alpha_1} \leq \underline{P}(Z_1), &
      & \sum_{i,j} \overline{P}(Z_1)_{ij} \le \beta_1
      \\
      &{-\alpha_2} \le \underline{R}(P), &
      &\sum_{i,j} \overline{R}(P)_{ij} \le \beta_2, 
      \,\forall P \in \mathcal{C}(\alpha_1, \beta_1)
      \\
      &{-\alpha_3} \le \underline{S}(R,Z_1), \ & 
      &\smash[l]{\sum_{i,j} \overline{S}(R,Z_1)_{ij} \le \beta_3},
      \,\forall R \in \mathcal{C}(\alpha_2, \beta_2)
      \\
      &Z_1 = G(U), &
      &\Delta U_{j+1} = \begin{cases} 
      U_{0} - \smash{U_{0}^{\prime}}, & j =0
      \\
      U_{j} - U_{j-1}, & j>0
      \end{cases}
      \end{alignedat}
    \end{align*}
    where $\Delta U_{j}$ is the $j$th element of $\Delta U$
    and $U_{0}^{\prime}$ is the first element of the optimal solution at the previous time step. Denote the optimal $(U,Y)$ as $(\smash{U^{k+1}},\smash{Y^{k+1}})$.
    \State \textbf{Update iteration counter:} $k \gets k + 1$.
\Until{
  $\|U^{k} - U^{k-1}\| \le \varepsilon$ or $k \ge K_{\max}$}.
\State \textbf{Output:} $U^\star \gets U^{k}$ and $Y^\star\gets Y^k$.
\end{algorithmic}
\end{algorithm}

\section{Convergence analysis}\label{sec:converge}
%
%
The convex subproblem in Algorithm~\ref{alg:MCCPA} (line~7) minimizes a quadratic cost penalizing an upper bound on the predicted tracking error relative to the output reference $Y^r$. This bound is computed using the sets $\mathcal{C}(\alpha_i,\beta_i)$, $i=1,2,3$ containing all admissible realizations of the internal variables of the TFE model. 
The constraints of the problem are imposed as convex constraints on the vertices of these simplexes.
Since the bounds derived in Section~\ref{sec:CCP} are tight, the choice of linearization points esure that, at each iteration a feasible (but suboptimal) solution is necessarily given by
$Z_1 = Z_1^k$ and $\alpha_i,\beta_i$, $i=1,2,3$ corresponding to
\[
\mathcal{C}(\alpha_1,\beta_1) = \{P^k\}, \
\mathcal{C}(\alpha_2,\beta_2) = \{R^k\}, \
\mathcal{C}(\alpha_3,\beta_3) = \{S^k\},
\]
and the objective value therefore satisfies, for all $k$,
\[
  J^{k+1} \leq J^{k}.
\]
A second key element of the method is that the constraints linking~$P$, $R$, and $S$ across stages form a consistent sequence of uncertainty sets. Moreover, because each convex subproblem is feasible by construction and solved exactly using a convex optimizer, as $k\to\infty$ we necessarily have
\[
  Z_1^{k+1} - Z_1^{k} \to 0,
  \qquad
  J^{k+1} - J^{k} \to 0,
\]
ensuring that, for each $i=1,2,3$, $\mathcal{C}(\alpha_i,\beta_i)$ converges to a singleton 
and the iteration of Algorithm~1 converges to a locally optimal solution of the tracking control problem \citep[see e.g.][Theorems 2 and 3]{ju_1}.

\section{Experimental results} \label{sec:results}

\begin{figure}[b]
\centering  
\includegraphics[trim={2mm 2mm 2mm 2mm},clip,scale=0.5]{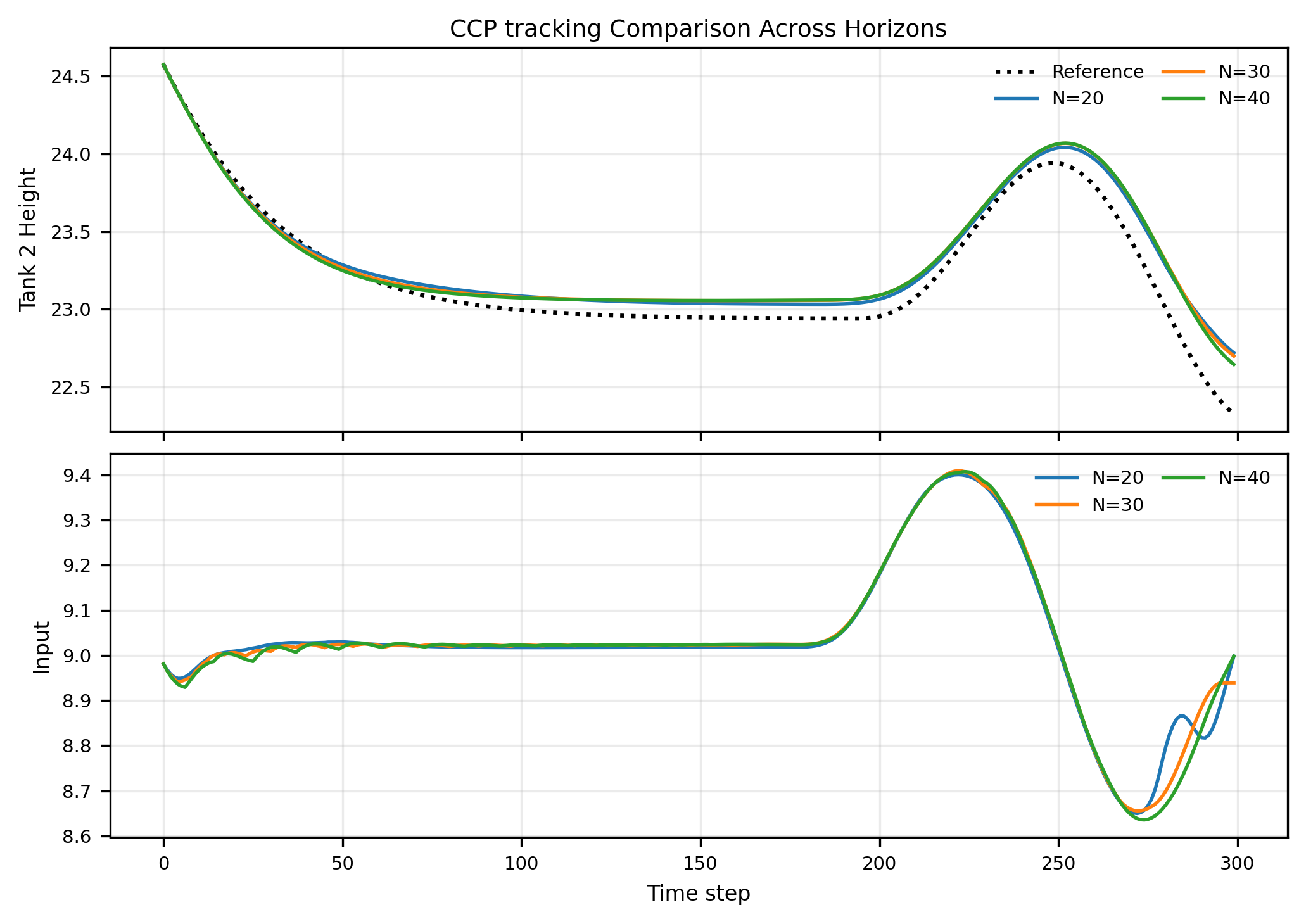} 
\caption{Closed-loop tracking performance of CCP-MPC under  different prediction horizons ($p=20$).}
\label{fig:CCP_MPC horizons}
\end{figure}
The closed-loop responses of CCP-based MPC are compared in Figure \ref{fig:CCP_MPC horizons} for prediction horizons $N=20$, $30$, and $40$, with control weighting $\lambda = 8$ in the MPC cost. 
Here $p=20$ is fixed, and for $N>p$ we therefore extend the lengths of vectors $Y$, $U$ to $N$ by recursively applying the TFE predictor with $p=20$.
The dominant dynamics are captured with $N=20$, and increasing the horizon to $30$ or $40$ does not produce significant improvements in closed-loop tracking performance, as shown in Table~\ref{tab:ccp_slsqp_timing_error}.

\begin{table}[b]
\centering
\caption{Mean ($t_{\mathrm{av}}$) and maximum ($t_{\max}$) solve times, and  closed-loop MPC tracking error (RMSE) for CCP and SLSQP with varying $N$.}
\label{tab:ccp_slsqp_timing_error}
\begin{tabular}{c|ccc|ccc}
\hline
{$N$} & \multicolumn{3}{c|}{CCP\rule{0pt}{8pt}} & \multicolumn{3}{c}{SLSQP} \\
 & $t_{\mathrm{av}}$\,(s) & $t_{\max}$\,(s) & RMSE & $t_{\mathrm{av}}$\,(s) & $t_{\max}$\,(s) & RMSE \\
\hline
20 & 0.075 & 0.203 & 0.130 & 0.046 & 0.118 & 0.355\rule{0pt}{8pt} \\
30 & 1.857 & 3.229 & 0.138 & 0.107 & 0.393 & 0.355 \\
40 & 5.166 & 9.421 & 0.133 & 0.100 & 0.257 & 0.355 \\
\hline
\end{tabular}
\end{table}
 With increasing prediction horizon $N$, the computational cost of CCP grows significantly, as can be seen from the mean and maximum computation times $t_{\mathrm{av}}$ and $t_{\max}$ in Table~\ref{tab:ccp_slsqp_timing_error}. In contrast, SLSQP maintains relatively low average solve times across all tested horizons. However, the tracking accuracy shows a clear advantage for CCP. The RMSE values obtained by CCP remain low and nearly constant across all horizons, ranging from 0.130 to 0.138. This indicates that although SLSQP requires less computation time, it produces less accurate tracking performance in this experiment. Therefore, CCP is preferred for its better tracking accuracy and more reliable solution quality.

\begin{figure}[t]
\centering  
\includegraphics[trim={3mm 3mm 1mm 2mm},clip,scale=0.5]{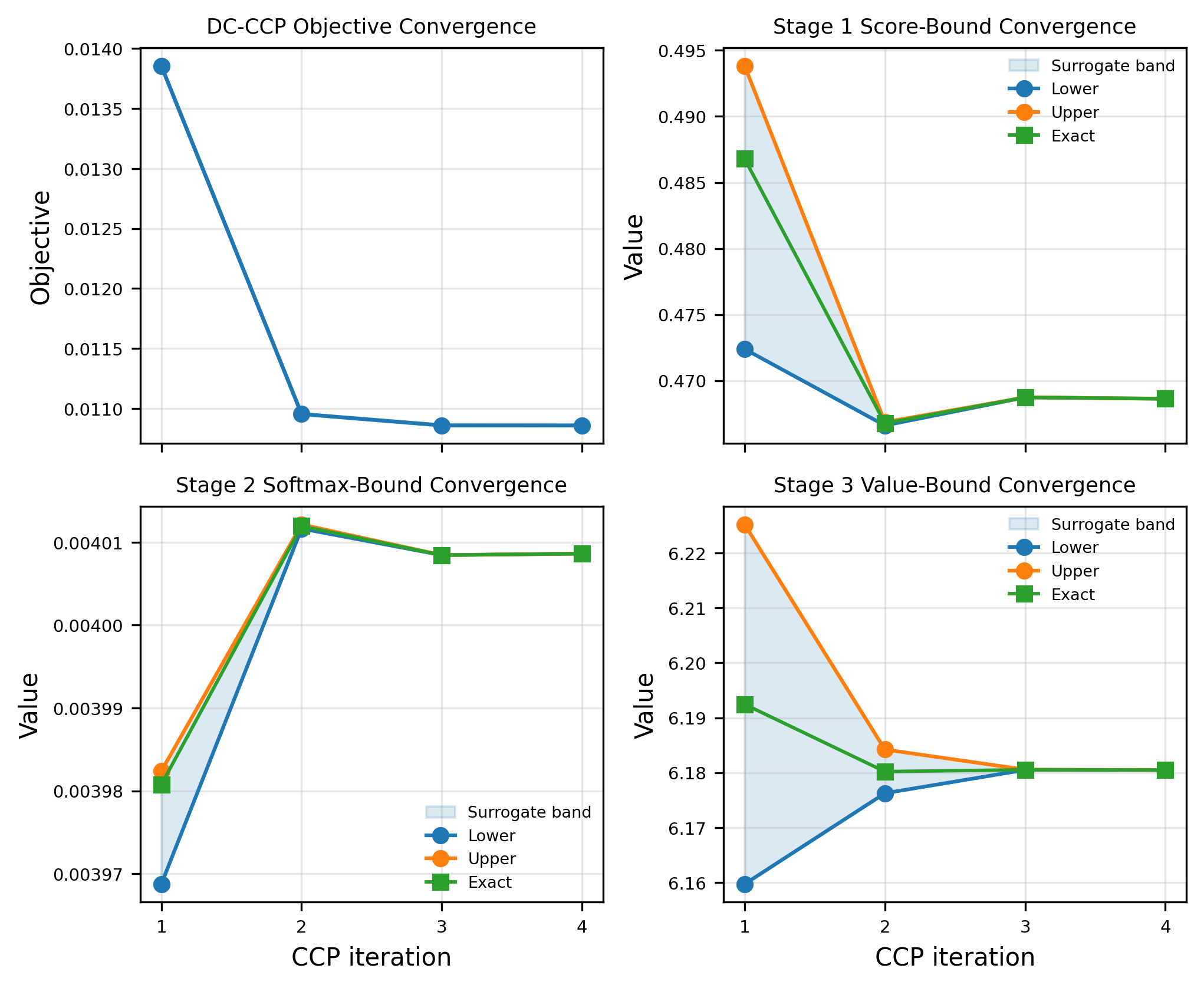} 
\caption{Convergence behavior of the proposed DC-CCP algorithm over CCP iterations.}
\label{fig:CCP_MPC bounds convergence}
\end{figure}
Figure \ref{fig:CCP_MPC bounds convergence} illustrates the bounds on the objective function and on representative entries of matrices $P$, $R$, $S$ in stages 1-3 of the TFE model at consecutive iterations of the CCP algorithm.
As expected from the analysis of Section~\ref{sec:converge}, the objective value converges after a few iterations, while the surrogate lower and upper bounds rapidly tighten around the optimal values of $P$, $R$, and $S$. This is consistent with the observation that Algorithm~1 reliably achieves convergence after just $3$ CCP iterations.

\section{Conclusions} \label{sec:conclu}
This paper investigates the integration of a transformer encoder–based multi-step predictor into an MPC framework, with online optimization carried out using the convex–concave procedure. The approach was evaluated on a benchmark nonlinear two-tank water level control problem. We derive convexification procedures tailored to key components of the transformer encoder architecture, with particular emphasis on the self-attention mechanism. Future work will focus on reducing the size of the DC program to enable longer prediction horizons and lower online computation, refining the DC approximations (e.g. via log-softmax relaxations), and extending the transformer encoder model to include multi-head attention.




\bibliography{reference}             
\end{document}